\documentclass[10pt]{article}
\usepackage{amsfonts,amssymb,latexsym, amsmath, amsthm}

\usepackage[usenames]{color}

\newtheorem{thm}{Theorem}[section]
\newtheorem{lem}[thm]{Lemma}
\newtheorem{cor}[thm]{Corollary}

\theoremstyle{remark}
\newtheorem{rem}[thm]{Remark}

\def\NN{{\mathbb N}}

\def\rst{{\rm rist}}

\def\leq{\leqslant}
\def\geq{\geqslant}
\def\leqf{\leq_{\rm f}}
\def\normf{{\,\triangleleft_{\rm f}\,}}
\def\fX{{\frak X}}

\def\Aut{{\rm Aut}}

\begin{document}
\title{\Large\bf On subgroups of finite index in branch groups}
\author{Alejandra Garrido and John S.\ Wilson\\   Mathematical Institute, Oxford, United Kingdom}

\date{}
\maketitle

\begin{abstract}
We give a structural description of the normal subgroups of subgroups of finite index in branch groups in terms of rigid stabilizers. This gives further insight into the structure lattices of branch groups introduced by the second author. We derive a condition concerning abstract commensurability of branch groups acting on the $p$-ary tree for any prime $p$.
\end{abstract}

\section{Introduction}
The class of branch groups was introduced by Grigorchuk in 1997, to provide a general framework for
studying groups arising as counter-examples in a wide variety of contexts.  This class also plays a natural role in the study of just infinite groups (see \cite{Jsurvey}). 
By now, the structure
theory of branch groups is quite well developed; see for example \cite{handbook, slavachapter, Jsurvey}.  Among the remarkable properties of arbitrary branch groups is a striking result proved by Grigorchuk describing their 
non-trivial normal subgroups; it has the consequence that every proper quotient of a branch group is virtually abelian.  The definition of branch groups (given below) shows that some subgroups of finite index (for example, subgroups $\rst_G(n)$ with $n>0$) do not share this property.
 However, our main theorem, Theorem \ref{main} below, gives a reasonably precise description of normal subgroups of subgroups of finite index in branch groups. 

Branch groups are defined in terms of their action on a specific type of tree.
 Let  $(m_n)_{n\geq0}$ be a sequence of integers with $m_n\geq2$ for each $n$. 
The rooted tree of type $(m_n)$ is a tree $T$ with a vertex $v_0$ (called the root vertex) of valency $m_0$, such that every vertex at a distance $n\geq1$ from $v_0$ has valency $m_n + 1$. 
The distance of a vertex $v$ from $v_0$ is called the {\em level} of $v$, and the set $L_n$ of vertices of level $n$ is called the $n$th {\em layer} of $T$. 
 Each vertex $v$ of level $r$ is the root of a rooted subtree $T_v$ of type $(m_{n})_{n\geq r}$. 
We picture $T$ with the root at the top and with $m_n$ edges descending from each vertex of level $n$.  Therefore we call the vertices below a vertex $v$ the {\em descendants} of $v$.
 If $m_n=d$ for every $n$, we say that $T$ is a $d$-ary tree.
 
Now suppose that $G$ is a group of tree automorphisms of $T$ fixing $v_0$. For each vertex $v$ write $\rst_G(v)$ for the subgroup of elements of $G$ that fix all vertices outside $T_v$, and for each $n\geq0$ write $\rst_G(n)$ for the direct product $\langle \rst_G(v)\mid v\in L_n\rangle$. We also write $\rst_G(X)=\prod_{x\in X} \rst_G(x)$ for each subset $X$ of  $L_n$.  
The group $G$ is said to be a {\em branch group} on $T$ if the following two conditions hold for each $n\geq0$:   \begin{enumerate}
\item[(i)] $G$ acts transitively on $L_n$;
\item[(ii)]  $|G : \rst_G(n)|$ is finite. \end{enumerate}

Notice that the transitive action of $G$ on all layers of $T$ implies that $G$ is infinite.  Moreover $G$ is evidently residually finite.  We shall use these facts  throughout the paper without further mention.

Branch groups are subject to strong restrictions. The proof of  \cite[Lemma 2]{geomded} shows that 
they have no non-trivial virtually abelian normal subgroups and in \cite[Theorem 4]{slavachapter} the following description of normal subgroups is given:

\begin{thm}\label{slava}  Suppose that $G$ is a branch group acting on a tree $T$ and let 
$K\triangleleft G$ with $K\neq1$.  Then $K$ contains the derived subgroup $\rst_G(n)'$ of
$\rst_G(n)$ for some integer $n$. \end{thm}

We shall prove the following result for normal subgroups of subgroups of finite index:

\begin{thm}\label{main} Suppose that $G$ is a branch group acting on a tree $T$,  let $H$ be a subgroup of finite index and let
$K\triangleleft H$.  
Then for all sufficiently large integers $n$ there is a union $X$ of $H$-orbits in $L_n$ such that
$$K\cap\rst_G(n)'=\rst_G(X)'.\eqno(1)$$  More precisely,
$$\rst_G(X)'\leq K\quad\hbox{and}\quad K\cap \rst_G(L_n\setminus X)=[K, \rst_G(L_n\setminus X)]=1.\eqno(2)$$ 
\end{thm}

The second part of (2) above is simply the statement that the subgroups $K$ and $\rst_G(L_n\setminus X)$ generate their direct product in $G$.

It will follow easily from this result  that for any subgroup $H$ of finite index of a branch group, the number of infinite $H$-invariant direct factors of a normal subgroup of $H$ is bounded  by $|G:H|$. These results allow us to establish in Section \ref{application} a necessary condition for direct products of certain branch groups acting on the $p$-ary tree to be (abstractly) commensurable, based on the number of direct factors.  This complements a result of the first author \cite{Atransfer} concerning the Gupta--Sidki $3$-group.

We sketch another application.  The structure lattice $\cal L$ and structure graph $\cal T$ of a branch group $G$ have been defined and studied in \cite{geomded, Jsurvey}.  They depend only on the structure of $G$ as an abstract group and not on its action on the tree $T$.
The structure lattice $\cal L$ is a Boolean lattice obtained from the set of all subgroups of $G$ which have finitely many conjugates by declaring two such subgroups equivalent if their centralizers in $G$ coincide. 
The structure graph $\cal T$ has as vertices the equivalence classes in $\cal L$ corresponding to {\em basal} subgroups. 
These are the subgroups whose finitely many conjugates generate their direct product (in particular, $\rst_G(v)$ is a basal subgroup for each $v \in T$). 
Defining the edges of $\cal T$ is slightly more technical, and we refer the reader to \cite{geomded}. 
In that paper it is also shown that in many important cases $\cal T$ is in fact a tree and is even isomorphic to $T$.
The action of $G$ on its subgroup lattice induces an action by tree automorphisms on $\cal T$;
thus we can find within the subgroup structure of $G$ the tree on which $G$ acts as a branch group  and the action itself.
Theorem \ref{main} gives another approach to these objects. 
Let ${\cal L}_0$ be the family of subgroups of the form $\rst_G(X)$
with $X$ an $H$-invariant subset of a layer of $T$ for some subgroup $H\leqf G$. 
For such subsets $X_1$, $X_2$ write $\rst_G(X_1)\sim \rst_G(X_2)$ if $X_2\subseteq L_n$ consists of all descendants of $X_1$ in $L_n$ (or vice-versa).
Then $\sim$ is an equivalence relation on ${\cal L}_0$ and the quotient set with naturally defined operations is a Boolean lattice isomorphic to  $\cal L$.
  Moreover $\cal T$ corresponds to the equivalence classes
containing subgroups $\rst_G(X)$ for orbits $X$ that do not split into unions of distinct $H$-orbits 
in lower layers.  We leave the details to the interested reader.

\section{Preliminaries and proof of Theorem \ref{main}}

\noindent{\bf Notation.} \  We write $H\leqf G$ (resp.\ $H\normf G$) to mean that $H$ is a subgroup (resp.\ a normal subgroup) of finite index in $G$, and $K^S$ denotes the subgroup of $G$ generated by all conjugates of a subgroup $K$ under a subset $S\subseteq G$. \medskip

We begin by noting an immediate consequence of the definition of branch groups. It will be used throughout the rest of the paper. \medskip

\begin{rem}\label{bounded orbits}  Let $G$ be a branch group on a tree $T$ with layers $L_n$.
Since $G$ acts transitively on each $L_n$, a subgroup $H\leq G$ of finite index will have at most $|G:H|$ orbits on $L_n$. 
If $X$ is an $H$-orbit in $L_n$ then the descendants of $X$ on $L_{n+1}$ comprise a union of $H$-orbits.
Hence there is some $n_0$ such that for all $n\geq n_0$
the number of $H$-orbits on $L_n$ is equal to the number of $H$-orbits on $L_{n_0}$.
\end{rem}\medskip

The following preliminary results concern virtually soluble subgroups and factors of branch groups and they will be used in the proof of Theorem \ref{main}.
The next lemma gives some consequences of a result from \cite{geomded}.

\begin{lem}\label{notvsol} Let $G$ be a branch group on a tree $T$.  
Then $G$ has no non-trivial abelian subgroups $K$ satisfying $K\triangleleft K^G\triangleleft G$.  Moreover, the following assertions hold\/$:$
\begin{enumerate} \item[\rm(a)]  $\rst_G(v)$ is not virtually soluble, for each vertex $v$ in $T;$ 
\item[\rm(b)]  if $H\leqf G$ then 
$H$ has no non-trivial virtually soluble normal subgroups. \end{enumerate} \end{lem}
\begin{proof}
The first assertion is \cite[Lemma 2]{geomded}.

(a)  The subgroup $\rst_G(v)$ is infinite since the direct product of its finitely many conjugates is some subgroup $\rst_G(n)$ and this has finite index in $G$.  Thus
if $\rst_G(v)$ is virtually soluble, then it must have a non-trivial abelian normal subgroup $K$.  Since then $K\triangleleft K^G\triangleleft G$ this gives a contradiction. 

(b) Now let $H\leqf G$. 
 First we claim that if $K$ is a finite normal subgroup of $H$ then $K=1$.
Let $C={\rm C}_H(K)$; then, since ${\rm N}_H(K)/C$ embeds in $\Aut(K)$, we have $C\leqf G$ and so $D\normf G$ where $D=\bigcap_{g\in G}C^g$.
 Moreover $D$ centralizes $K^G$, thus $D\cap K^G$ is an abelian normal subgroup of $G$.
  From above $D\cap K^G=1$ and so $K^G$ is finite. 
 If $K\neq1$ then by Theorem \ref{slava} we have $\rst_G(n_1)'\leq K^G$ for some $n_1$.   
Because $G$ is residually finite, there is some $L\normf G$ with $L\cap K^G=1$, and 
we have $\rst_G(n_2)'\leq L$ for some $n_2$. 
 But then $\rst_G(n)$ must be abelian where $n=\max(n_1,n_2)$,  and this is a contradiction.

Finally suppose that $1\neq K\triangleleft H$ and that $K$ is virtually soluble.  The soluble normal subgroup $K_0$ of $K$ of smallest index is normal in $H$ and non-trivial from above,
and the last non-trivial term $A$ of the derived series of $K_0$ is an abelian normal subgroup of $H$.
From above, $A$ is infinite and so $A\cap M$ is a non-trivial abelian normal subgroup of $M$ where
$M=\bigcap_{g\in G}H^g$.  But this gives another contradiction to the first assertion of the lemma.
\end{proof}

Our next result is a variant of an argument in \cite{Jsfirst}.  In our application of it we take for $\fX$ the class of virtually soluble groups.  This is evidently quotient-closed and it is also extension-closed.

\begin{lem}\label{general X} Let $\fX$ be a class of groups that is closed for quotients and extensions, and let $H\normf G$.
If each ascending chain $(K_i)_{i\in\NN}$ of normal subgroups of $G$ has at most $c$
factors $K_i/K_{i-1}$ that are not in $\fX$, then each ascending chain of 
normal subgroups of $H$ has at most $c\cdot2^{|G:H|-1}$ factors that are not in $\fX$.
\end{lem}

\begin{proof}  
We begin by showing that if $A_1,A_2,B_1,B_2$ are normal subgroups of $H$ with 
$A_1\leq A_2$ and $B_1\leq B_2$, and if both $A_2B_2/A_1B_1$, $(A_2\cap B_2)/(A_1\cap B_1)$ are in $\fX$, then $A_2/A_1$ is in $\fX$. 
The quotient $A_2B_2/A_1B_2$ is in $\fX$ as it is isomorphic to a quotient of $A_2B_2/A_1B_1$. 
Similarly, $(A_2\cap B_2)/ (A_1\cap B_2)$ is also in $\fX$. 
The claim follows since $A_2/A_1$ is an extension of $A_2/(A_2\cap A_1B_2)\cong A_2B_2/A_1B_2$ by $(A_2\cap A_1B_2)/A_1 \cong (A_2\cap B_2)/(A_1\cap B_2)$.

Write $n=|G:H|$.  For the purposes of this proof, we call a non-empty finite subset $S$ of $G$ {\em good} if whenever $(K_i)$ is an ascending chain of normal subgroups of $H$ then the chain
$(K_i^S)$ has at most $c\cdot2^{n-|S|}$ factors that are not in $\fX$.   Thus if $S$ is a transversal 
to $H$ in $G$ then $S$ is good, and our conclusion holds if and only if $\{1\}$ is good.      
Choose a good set $S$ of smallest cardinality.  Then each set $Ss_0^{-1}$ with $s_0\in S$ is good, and so we may assume that $1\in S$.   Suppose that $S\neq\{1\}$ and write
$S'=S\setminus\{1\}$;
 thus $S'\neq\emptyset$.  Let $(K_i)$ be an ascending chain of normal subgroups of $H$ and write $K_i^\ast=K_i\cap K_i^{S'}$ for each $i$.
Since $S$ is good, the set $J$ of indices $j$ for which either $K_j^S/K_{j-1}^S$ or
$K_j^{\ast S}/K_{j-1}^{\ast S}$ is not in $\fX$ has at most
$2(c\cdot 2^{n-|S|})=c\cdot  2^{n-|S'|}$ elements.
However it is easy to check that
$$K_i^{S'}(K_i^{\ast\,S'}K_i)=K_i^S\quad\hbox{and}\quad K_i^{S'}\cap(K_i^{\ast\,S'}K_i) =K_i^{\ast\,S}$$ for each $i$.  Therefore, for the indices $i\notin J$, the claim in the first paragraph yields that $K_i^{S'}/K_{i-1}^{S'}\in \fX$.
But this shows that the set $S'$ is good, and the result follows from this contradiction.
 \end{proof}
 
This lemma has the following consequence which will be necessary for the proof of Theorem \ref{main}.
 
 \begin{lem}\label{short chains}  Let $G$ be a branch group acting on a tree $T$, and $H\normf G$.   Then in any series of normal subgroups of $H$
there are at most $2^{|G:H|-1}$ factors that are not virtually soluble.
 \end{lem}
 
 \begin{proof}  By Theorem \ref{slava}, every proper quotient of $G$ is virtually abelian.
Therefore the result holds from Lemma \ref{general X}, with $c=1$ and $\fX$ the class of
virtually soluble groups.
\end{proof}

We are now ready to prove the main theorem. For the reader's convenience, we present an important special case separately.

\begin{lem}\label{all rst big} Let $G$ be a branch group acting on a tree $T$, and $K\triangleleft H\leqf G$. 
If $K\cap \rst_G(v)>1$ for all vertices $v$ in $T$ then $\rst_G(n)'\leq K$ for some $n$.
\end{lem}

\begin{proof}  Choose $n_0$ as in Remark \ref{bounded orbits} and also with $\rst_G(n_0)'\leq H$.
Let $X_1,\dots,X_r$ be the orbits of $H$ on $L_{n_0}$.

Fix $i$ and let $v$ be a vertex in $X_i$.
Choose $k\in K$ and $m_i\geq0$ such that $k$ does not fix
the $m_i$th layer in $T_v$; choose $u$ in this layer with $uk\neq u$.
 We claim that $\rst_G(u)'\leq K$.  
Let $x,y\in\rst_G(u)$; thus
$y^k\in\rst_G(uk)$ and $[x,y^k]=1$.  
Hence $[x,[k,y]]= [x,(y^k)^{-1}y]= [x,y]$.
 Since $[x,[k,y]]\in K$ our claim follows. 
 
 Because $H$ normalizes $K$ it follows that $K$ contains $\rst_G(Y)'$ where
$Y$ is the $H$-orbit in $L_{n_0+m_i}$ containing $u$.
  Arguing thus for each $i$, we conclude that $\rst_G(n)'\leq K$ where $n=n_0+\max(m_1,\dots,m_r)$.  
\end{proof}

\begin{proof}[Proof of Theorem  {\rm \ref{main}}]    We begin by noting that the statements in (2) imply equation (1).  If (2) holds then  we have
$$\begin{array}{rl} K\cap\rst_G(n)'&=K\cap (\rst_G(X)'\times \rst_G(L_n\setminus X)'\,)\\
&= \rst_G(X)'\times(K\cap \rst_G(L_n\setminus X)').\end{array}$$
From (2), the second factor here is trivial.  

Therefore it suffices to prove (2).  For each normal subgroup $K$ of $H$, every chain of 
normal subgroups of $H$ from $K$ to $H$ with no virtually soluble factors has length 
bounded independently of $K$, by Lemma \ref{short chains}.  Write $j(K)$ for the maximal length of such a chain.  

Suppose that the result is false. 
Pick a subgroup $K$ demonstrating this and with $j(K)$ as small as possible.  
By Lemma \ref{all rst big} we may assume $\rst_G(v)\cap K=1$ for some vertex $v$.  
Since $H$ contains a normal subgroup of finite index, it contains $\rst_G(n_1)'$ for some $n_1\in\NN$.  Let $w$ be 
a vertex in $T_v$, and with $w\in L_{n_2}$ where $n_2\geq n_1$. 
Let $X_1$ be the $H$-orbit of $w$.  

The subgroups $\rst_G(w)'$ and $K\cap\rst_G(n_2)'$ are disjoint 
and both normal in $\rst_G(n_2)$, and hence they commute.  But $K\triangleleft H$; so
$\rst_G(u)'$ and $K\cap\rst_G(n_2)'$ commute for all $u\in X_1$, and hence
$\rst_G(X_1)'$ and $K\cap\rst_G(n_2)'$ commute.  Therefore 
$K$ and $\rst_G(X_1)'$  intersect in a
virtually abelian normal subgroup of $H$, so intersect trivially and commute by Lemma \ref{notvsol}.
It follows that these
subgroups generate their direct product $K_1$.  Since $K_1\triangleleft H$ and $K_1/K\cong \rst_G(X_1)'$, which is not virtually soluble by Lemma \ref{notvsol} (a), we have $j(K_1)<j(K)$.  Therefore we can find some $n_0\geq n_2$ such that for every $n\geq n_0$ there is a
union $X_2$ of $H$-orbits of $L_n$ satisfying 
$$\rst_G(X_2)'\leq K\times \rst_G(X_1)'\quad\hbox{and}\quad (K\times (\rst_G(X_1)')\cap \rst_G (L_n\setminus X_2)=1.\eqno(\ast)$$ 
Let $Y_1$ be the set of descendants of $X_1$ in $L_n$; thus
$\rst_G(Y_1)'\leq \rst_G(X_1)'$ and hence $Y_1\subseteq X_2$.  Let $X=X_2\setminus Y_1$;
this is a union of $H$-orbits of $L_n$.  The first inequality in $(\ast)$ gives 
$$\rst_G(X)'\leq \rst_G(Y_1)'\times K,$$ and 
since $\rst_G(X)$, $\rst_G(Y_1)$ commute it follows that
$$(\rst_G(X))''\leq[ K\times\rst_G(Y_1)',\rst_G(X)]\leq[K,\rst_G(X)]\leq K.$$   
Thus the image of $\rst_G(X)'$ under the projection map
$\rst_G(Y_1)'\times K\to \rst_G(Y_1)'$ is abelian and $H$-invariant, and so is trivial from
Lemma \ref{notvsol}.  Hence $\rst_G(X)'\leq K$.  
The second inequality  in $(\ast)$ gives $K\cap \rst_G (L_n\setminus X_2)=1$; thus $K\cap \rst_G(n)$
commutes with $\rst_G(L_n\setminus X_2)$ as well as $\rst_G(Y_1)$, and so commutes with 
$\rst_G(L_n\setminus X)$.  The result follows.  
\end{proof}
We can now strengthen Lemma \ref{short chains}.

\begin{cor}\label{shorter chains}  Let $G$ be a branch group acting on a tree $T$, and $H\leqf G$.   Then in any series of normal subgroups of $H$
there are at most $|G:H|$ factors that are not virtually abelian.
 \end{cor}

\begin{proof}  
For an $H$-invariant subset $X$ of a layer $L_n$ of $T$ we write $o(X)$ for the number of
orbits of $H$ on $X$.

Let $(K_i)_{i\in\mathbb{N}}$ be an ascending series of normal subgroups of $H$.
Then, by Theorem \ref{main}, for each $i$ there exist some $n_i$ and some union $X_i$ of $H$-orbits in $L_{n_i}$ such that $\rst_G(n_i)'\leq H$ and $\rst_G(n_i)'\cap K_i=\rst_G(X_i)'$.
Moreover $n_i$ can be chosen so that $o(X_i)$ is as large as possible (by Remark \ref{bounded orbits}). 
Clearly $o(X_{i-1})\leq o(X_i)$ for each $i$. 
 Thus there are at most $|G:H|$ indices $i$ for which $o(X_{i-1})<o(X_i)$.  
Suppose that $j\in \NN$ is not one of these indices. 
 We claim that the quotient $K_j/K_{j-1}$ is virtually abelian. 
To see this, let $n\geq n_{j-1},n_j$, replace $X_{j-1}$ and $X_j$ by their respective sets of descendants in $L_n$, set $R=\rst_G(n)'$ and notice that
$$\rst_G(X_{j-1})'=R\cap K_{j-1}\leq R\cap K_j=\rst(X_j)'.$$
Since $o(X_{j-1})=o(X_j)$, equality holds above.
Therefore the kernel of the obvious homomorphism from $K_j$ to $H/RK_{j-1}$ equals
$K_j\cap(RK_{j-1})=K_{j-1}$.  By Theorem \ref{slava}, the quotient $H/R$ is virtually abelian, and so therefore are $H/RK_{j-1}$ and $K_j/K_{j-1}$.
Our claim follows, and the result is proved.

\end{proof}

\section{An application: abstract commensurability}\label{application}

We recall that two groups $G_1$, $G_2$ are said to be abstractly commensurable if 
they have isomorphic subgroups of finite index.   Many branch groups $G$ have the property that $G$ is abstractly commensurable with the direct product of $n$ copies of $G$ for some integer $n>1$.  Given such a group $G$, it is natural to ask for which integers $n$ this holds.  Here we address a natural extension of this question.                                                                                               
 \medskip

\noindent {\bf Definition.} For an infinite group $H$, let $b(H)$ be the largest number
$r$ such that $H$ has $r$ infinite normal subgroups that generate their direct product in $H$;
if no such $r$ exists write $b(H)=\infty$. \medskip

If $H$ is a subgroup of finite index in a branch group (or if $H$ is any infinite group having no non-trivial finite normal subgroups) then $b(H)$ is the maximal number of factors in an irredundant subdirect product decomposition of $H$. 

\begin{cor}\label{invariant factors}
Let $G$ be a branch group acting on a tree $T$ and let $H\leqf G$. Then $b(H)$ is finite and is the maximum number of $H$-orbits on any layer of $T$.
\end{cor}
\begin{proof}
By Theorem \ref{slava} and Remark \ref{bounded orbits}, there exist $n_0$ and some $m\leq|G:H|$ such that
$\rst_G(n_0)' \leq H$ and such that $H$ has $m$ orbits on $L_n$ for all $n\geq n_0$.
If $X_1,\dots,X_m$ are the orbits of
$H$ on $L_{n_0}$ then the $m$ subgroups $\rst_G(X_i)'$ are infinite
normal subgroups of $H$ and generate their direct product.

Now let $K_1,\ldots, K_r$ be infinite normal subgroups of $H$ that generate their direct product. By Theorem \ref{main}, for each $i$ we can find an integer $n_i\geq n_0$ and a union $X_i$ of $H$-orbits on $L_{n_i}$ such that $\rst_G(X_i)'\leq K_i$.
Let $n=\max \{n_1,\dots, n_r\}$ and for each $i$ choose an $H$-orbit $Y_i$ of $L_n$
consisting of descendants of elements of $X_i$. Obviously the sets $Y_i$ are disjoint and so $r\leq m$.
\end{proof}

\begin{lem}\label{mod p-1}
Let $T$ be the $p$-ary tree for some prime $p$ and let $H$ be a subgroup of $\Aut (T)$ that acts on each layer as a $p$-group.
Then the number of $H$-orbits on each layer is congruent to $1$ modulo $p-1$.
In particular, if $H$ is also a subgroup of finite index in a branch group on $T$, then $b(H) \equiv 1 \mod{p-1}$.  
\end{lem}

\begin{proof}  This is elementary.  
Pick some layer $L_n$. Then the size of each $H$-orbit in $L_n$ is a power of $p$.  Writing $s_i$ for the number of orbits of size $p^i$ for $i=0,\ldots,n$ we have $p^n=s_0+ps_1+\cdots + p^ns_n$, so we obtain $1\equiv s_0+s_1+\cdots + s_n\mod{p-1}$.
\end{proof}

Using these results we show that direct products of certain branch groups on $p$-ary trees can only be (abstractly) commensurable if the numbers of direct factors are congruent modulo $p-1$. 

\begin{lem}\label{subdirect}
Let $\mathfrak{C}$ be the class of all groups $G$ with no non-trivial abelian normal subgroups and  with $b(G)$ finite.
Let $H_1,\ldots,H_n$ be groups in $\mathfrak{C}$,
and $H$ a subgroup of finite index in the direct product $D=H_1\times \cdots\times H_n$
such that each projection $\rho_j:H\rightarrow H_j$ is surjective.
 Then $b(H)=b(H_1)+\cdots +b(H_n)$.
\end{lem}

\begin{proof}

Let $K_1,\dots, K_r$ be infinite normal subgroups of $H$ that generate their direct product in $H$.
For $i\in \{1,\ldots,r\}$ and $j\in\{1,\ldots,n\}$, write $K_{i,j}=\rho_j(K_i)\triangleleft H_j$.
If $i\in \{1,\ldots,r\}$ and $K_{i,j}$ is finite for all $j$ then $K_i$ is finite, a contradiction.
So for each $i$ there is some $j$ such that $K_{i,j}$ is infinite.
Therefore $r\leq\sum r_j$, where for each $j$, we write $r_j$ for the number of indices $i$
with $K_{i,j}$ infinite. Now each subgroup $K_i$ centralizes all other subgroups
$K_{i'}$ with $i'\neq i$. Therefore for fixed $j$, each $K_{i,j}$ centralizes the product $P_{i,j}$ of
all subgroups $K_{i',j}$ with $i'\neq i$; and since $H_j$ has no non-trivial abelian normal subgroups,
the normal subgroups $K_{i,j}$, $P_{i,j}$ intersect trivially. Thus the subgroups $K_{i,j}$ with
$1\leq i\leq r$ generate their direct product in $H_j$. It follows that $r_j\leq b(H_j)$, and hence
$r\leq b(H_1)+\cdots+b(H_n)$. In particular, $b(H)$ is finite and bounded by $\sum b(H_j)$.

It remains to prove that $\sum b(H_j)\leq b(H)$.  For each $j$, find a family of
$b(H_j)$ infinite normal subgroups of $H_j$ that generate their direct product in $H_j$.
If $L\triangleleft H_j$ is one of these subgroups, then its image $\bar L$ under
the natural injection $H_j\rightarrow D$ is an infinite normal subgroup of $D$; hence $\bar L\cap H$
is an infinite normal subgroup of $H$. In this way we find $\sum b(H_j)$ infinite normal subgroups of $H$ that evidently generate their direct product. 
The result follows.
\end{proof}

\begin{cor}\label{commimplies}
Let $\mathfrak{D}$ be the family of branch groups which act  as a $p$-group on each layer of the $p$-ary tree $T$. 
Let $\Gamma_1$ and $\Gamma_2$ be, respectively, direct products of $n_1$ and $n_2$ groups in $\mathfrak{D}$. If $\Gamma_1$ and $\Gamma_2$ are abstractly commensurable then $n_1\equiv n_2 \mod p-1$.
\end{cor}

\begin{proof}
By Corollary \ref{invariant factors}, $\mathfrak{D}$ is contained in $\mathfrak{C}$. Thus, for $i=1$, $2$, Lemmas \ref{mod p-1} and \ref{subdirect} together imply that 
$b(H_i)\equiv n_i \mod p-1$ for all $H_i\leqf \Gamma_i$.  Hence if there are isomorphic subgroups $H_1$, $H_2$ then we must have $n_1\equiv n_2\mod p-1$.
\end{proof}

Finally, let $\Gamma$ be the Gupta--Sidki $3$-group.
This is a branch group on the ternary tree, and its properties were investigated in  
\cite{GS, Sidki1, Sidki2}.  We note that {\em $\Gamma$ is an example of a finitely generated torsion group that is abstractly commensurable with its direct cube but not its direct square.}  The final assertion here follows from Corollary \ref{commimplies}.  
In \cite{Atransfer}, the first author proved that any infinite finitely generated subgroup
of $\Gamma$ is abstractly commensurable with either $\Gamma$ or $\Gamma\times\Gamma$.
Therefore {\em the Gupta--Sidki $3$-group has precisely two abstract commensurability classes of finitely generated infinite subgroups. }

\subsection*{Acknowledgement}
The first author thanks the {\em Fundaci\'on La Caixa} (Spain) for financial support.


\begin{thebibliography}{100}

\bibitem{handbook}  L. Bartholdi,  R. I. Grigorchuk and Z. Sunik, Branch groups. In {\em Handbook of Algebra}, vol. 3 (North-Holland, Amsterdam, 2003).  
\bibitem{Atransfer} A. Garrido. Abstract commensurability and the Gupta--Sidki group, Transfer dissertation, University of Oxford (2012).
\bibitem{slavachapter} R. I Grigorchuk.  Just infinite branch groups.  In {\em New Horizons in pro-$p$ Groups} (Birkh\"auser 2000), 121--179.
\bibitem{geomded} R.I. Grigorchuk and J.S. Wilson,  The uniqueness of the actions of certain branch groups on rooted trees, Geom. Dedicata {\bf100} (2003), 103--116.
\bibitem{GS} N.D. Gupta and S. Sidki,
On the Burnside problem for periodic groups, 
Math. Z. {\bf182} (1983), 385--388.
\bibitem{Sidki1} S. Sidki.  On a 2-generated infinite $3$-group: The presentation problem. 
J. Algebra {\bf110} (1987), 13--23. 
\bibitem{Sidki2} S. Sidki. 
On a 2-generated infinite $3$-group: Subgroups and automorphisms. 
J. Algebra {\bf110} (1987), 24--55. 
\bibitem{Jsfirst} J.S. Wilson.  Some properties of groups inherited by normal subgroups of finite index.
Math.\ Z. {\bf114} (1970), 19--21.
\bibitem{Jsurvey} J.S. Wilson.   Structure theory for branch groups.  In {\em Geometric and Homological Topics in Group Theory}, London Math. Soc. Lecture Note Ser. 358 (Cambridge University Press, 2009), 306--320.  


\end{thebibliography}
\end{document}